\documentclass[11pt]{article}
\def\thetitle{Rigidity of expanders and pseudorandom graphs}
\def\thepdftitle{Rigidity of expanders and pseudorandom graphs}

\usepackage{graphicx}

\usepackage{amsmath,amssymb}

\usepackage[usenames,dvipsnames,svgnames,table]{xcolor}
\definecolor{CombinatoricaAqua}{HTML}{00698C}
\definecolor{CombinatoricaBlue}{HTML}{3A3293}
\definecolor{CombinatoricaBrown}{HTML}{66220C}
\definecolor{CombinatoricaRed}{HTML}{DF2A27}
\definecolor{HarvardCrimson}{rgb}{0.6471, 0.1098, 0.1882}
\definecolor{DAGreen}{HTML}{339900}

\makeatletter
\let\reftagform@=\tagform@
\def\tagform@#1{\maketag@@@
	{(\ignorespaces\textcolor{CombinatoricaBrown}{#1}\unskip\@@italiccorr)}}
\renewcommand{\eqref}[1]{\textup{\reftagform@{\ref{#1}}}}
\makeatother

\usepackage[backref=page]{hyperref}
\hypersetup{%
	unicode,
	pdfencoding=auto,
	pdfauthor={Michael Krivelevich, Alan Lew, and Peleg Michaeli},
	pdftitle={\thepdftitle},
	pdfsubject={},
	pdfkeywords={},
	colorlinks=true,
	citecolor=CombinatoricaBlue,
	linkcolor=CombinatoricaAqua,
	anchorcolor=CombinatoricaBrown,
	urlcolor=HarvardCrimson}

\usepackage{amsthm}
\usepackage{thmtools}
\usepackage{bbm}
\usepackage{enumitem}
\usepackage[capitalize]{cleveref}
\Crefname{mainthm}{Theorem}{Theorems}
\Crefname{fact}{Fact}{Facts}
\Crefname{claim}{Claim}{Claims}

\declaretheoremstyle[
spaceabove=\topsep, spacebelow=\topsep,
headfont=\color{CombinatoricaBrown}\normalfont\bfseries,
bodyfont=\itshape,
]{thm}
\declaretheoremstyle[
spaceabove=\topsep, spacebelow=\topsep,
headfont=\color{CombinatoricaBrown}\normalfont\bfseries,
bodyfont=\normalfont,
]{dfn}
\declaretheoremstyle[
spaceabove=0.5\topsep, spacebelow=0.5\topsep,
headfont=\color{CombinatoricaBrown}\normalfont\bfseries,
bodyfont=\normalfont,
]{rmk}

\declaretheorem[style=thm,parent=section]{theorem}
\declaretheorem[style=thm,name=Theorem,sibling=theorem]{mainthm}
\declaretheorem[style=thm,sibling=theorem]{lemma}
\declaretheorem[style=thm,sibling=theorem]{corollary}

\declaretheorem[style=rmk,numbered=no]{remark}

\declaretheorem[style=definition,numbered=no]{acknowledgements}
\declaretheorem[style=definition,sibling=theorem]{definition}

\usepackage[nobysame,msc-links,non-sorted-cites]{amsrefs}

\renewcommand{\eprint}[1]{\href{https://arxiv.org/abs/#1}{arXiv:#1}}

\BibSpec{book}{%
	+{}  {\PrintPrimary}                {transition}
	+{,} { \textbf}                     {title} 
	+{.} { }                            {part}
	+{:} { \textit}                     {subtitle}
	+{,} { \PrintEdition}               {edition}
	+{}  { \PrintEditorsB}              {editor}
	+{,} { \PrintTranslatorsC}          {translator}
	+{,} { \PrintContributions}         {contribution}
	+{,} { }                            {series}
	+{,} { \voltext}                    {volume}
	+{,} { }                            {publisher}
	+{,} { }                            {organization}
	+{,} { }                            {address}
	+{,} { \PrintDateB}                 {date}
	+{,} { }                            {status}
	+{}  { \parenthesize}               {language}
	+{}  { \PrintTranslation}           {translation}
	+{;} { \PrintReprint}               {reprint}
	+{.} { }                            {note}
	+{.} {}                             {transition}
	+{}  {\SentenceSpace \PrintReviews} {review}
}
\BibSpec{incollection}{%
  +{}  {\PrintAuthors}                {author}
  +{,} { \textit}                     {title}
  +{.} { }                            {part}
  +{:} { \textit}                     {subtitle}
  +{,} { \PrintContributions}         {contribution}
  +{,} { \PrintConference}            {conference}
  +{}  {\PrintBook}                   {book}
  +{,} { }                            {booktitle}
	+{}  { \PrintEditorsB}              {editor}
	+{,} { }                            {publisher}
  +{,} { \PrintDateB}                 {date}
  +{,} { pp.~}                        {pages}
  +{,} { }                            {status}
  +{,} { \PrintDOI}                   {doi}
  +{,} { available at \eprint}        {eprint}
  +{}  { \parenthesize}               {language}
  +{}  { \PrintTranslation}           {translation}
  +{;} { \PrintReprint}               {reprint}
  +{.} { }                            {note}
  +{.} {}                             {transition}
  +{}  {\SentenceSpace \PrintReviews} {review}
}
\BibSpec{misc}{%
    +{}  {\PrintAuthors}                {author}
    +{,} { \textit}                     {title}
    +{.} { }                            {part}
    +{:} { \textit}                     {subtitle}
    +{,} { \PrintContributions}         {contribution}
    +{.} { \PrintPartials}              {partial}
    +{,} { }                            {journal}
    +{}  { \textbf}                     {volume}
    +{}  { \PrintDatePV}                {date}
    +{,} { \issuetext}                  {number}
    +{,} { \eprintpages}                {pages}
    +{,} { }                            {status}
    +{,} { \url}                        {url}
    +{,} { \PrintDOI}                   {doi}
    +{,} { available at \eprint}        {eprint}
    +{}  { \parenthesize}               {language}
    +{}  { \PrintTranslation}           {translation}
    +{;} { \PrintReprint}               {reprint}
    +{.} { }                            {note}
    +{.} {}                             {transition}
    +{}  {\SentenceSpace \PrintReviews} {review}
}

\makeatletter
\def\mathcolor#1#{\@mathcolor{#1}}
\def\@mathcolor#1#2#3{%
	\protect\leavevmode
	\begingroup
	\color#1{#2}#3%
	\endgroup
}
\makeatother
\definecolor{Red}{rgb}{0.618,0,0}
\definecolor{Blue}{rgb}{0,0,1}
\definecolor{Green}{rgb}{0,0.298,0}

\newcommand{\iso}{\mathrm{i}}

\usepackage{sectsty}
\sectionfont{\color{CombinatoricaBrown}}
\subsectionfont{\color{CombinatoricaBrown}}
\subsubsectionfont{\color{CombinatoricaBrown}}

\usepackage{soul}
\soulregister\ref7

\usepackage{pifont}
\usepackage{calc}

\usepackage[
	a4paper,
	left=1in,
	right=1in,
	top=1in,
	bottom=1in
]{geometry}

\title{\thetitle}

\author{
  Michael Krivelevich\thanks{
    School of Mathematical Sciences,
    Tel Aviv University,
    Tel Aviv 6997801, Israel.
    Email: \href{mailto:krivelev@tauex.tau.ac.il}
                {\tt krivelev@tauex.tau.ac.il}.
    Research supported in part by NSF-BSF grant 2023688.
  }
  \and
  Alan Lew\thanks{
    Faculty of Mathematics, Technion Israel Institute of Technology, Technion City, Haifa 3200003, Israel;
    Email: \href{mailto:alanlew@technion.ac.il}
                {\tt alanlew@technion.ac.il}.
  }
  \and
  Peleg Michaeli\thanks{
    Mathematical Institute,
    University of Oxford,
    Oxford, UK.
    Email: \href{mailto:peleg.michaeli@maths.ox.ac.uk}
                {\tt peleg.michaeli@maths.ox.ac.uk}.
    Research supported by ERC Advanced Grant 883810.
    For the purpose of Open Access, the author has applied a CC BY public
    copyright licence to any Author Accepted Manuscript version arising from
    this submission.
  }
}

\makeatletter
\def\namedlabel#1#2{\begingroup
  #2%
  \def\@currentlabel{#2}%
  \phantomsection\label{#1}\endgroup
}
\makeatother

\usepackage{mleftright}
\mleftright
\newcommand{\defn}[1]{{\bfseries #1}}

\newcommand{\eps}{\varepsilon}
\renewcommand{\phi}{\varphi}

\newcommand{\RR}{\mathbb{R}}

\newcommand{\cG}{\mathcal{G}}

\newcommand{\cP}{\mathcal{P}}

\newcommand{\sm}{\smallsetminus}
\newcommand{\es}{\varnothing}

\newcommand{\floor}[1]{\left\lfloor{#1}\right\rfloor}
\newcommand{\ceil}[1]{\left\lceil{#1}\right\rceil}

\newcommand{\vect}{\mathbf}

\newcommand{\pr}[0]{\mathbb{P}}
\newcommand{\E}[0]{\mathbb{E}}

\newcommand{\whp}[0]{\textbf{whp}}

\newcommand{\p}{\vect{p}} 

\usepackage{tabto}

\usepackage{comment}

\usepackage{subcaption}
\usepackage{tikz,pgfplots}
\pgfplotsset{compat=1.16}
\usetikzlibrary{calc}

\begin{document}
\date{}
\maketitle

\begin{abstract}
  A graph $G=(V,E)$ is called $d$-rigid if, for a generic embedding of its vertices in $\mathbb{R}^d$, the only continuous motions of the vertices preserving the distances between all pairs of adjacent vertices are those induced from the isometries of $\mathbb{R}^d$ (that is, translations and rotations of the whole graph). In this paper, we study rigidity properties of pseudorandom graphs.
   First, we consider $C$-expander graphs, a class of graphs recently studied in the context of Hamiltonicity of pseudorandom graphs. These are $n$-vertex graphs for which every vertex set $A$ of size smaller than $n/(2C)$ has a neighbourhood of size at least $C |A|$, and for every pair of disjoint sets $A,B$ of size at least $n/(2C)$ each, there is at least one edge between $A$ and $B$. We show that for every $C\ge 8$ and every integer $n\ge 9C$, every $n$-vertex $C$-expander is $\lfloor C/8\rfloor$-rigid.

  Next, we study $(n,r,\lambda)$-graphs, which are $n$-vertex $r$-regular graphs whose non-trivial adjacency eigenvalues are bounded in absolute value by $\lambda$. This is a well-known family of graphs, known to possess various pseudorandom properties. We prove that there exist absolute constants $c_1,c_2>0$ such that every $(n,r,\lambda)$-graph $G$ with $\lambda\le c_1 r$ is $\lfloor c_2 r\rfloor$-rigid.

  Our results are sharp up to the value of the universal constants involved, and they improve and extend previous work by the authors on the rigidity of random and pseudorandom graphs.
\end{abstract}

\section{Introduction}

A \defn{$d$-dimensional framework} is a pair $(G,\p)$, where $G=(V,E)$ is a graph and $\p:V\to\RR^d$ is an embedding of its vertices.
We say that $(G,\p)$ is \defn{rigid} if every continuous motion of the vertices that preserves the lengths of all edges preserves the distances between all pairs of vertices.
An embedding $\p:V\to\RR^d$ is \defn{generic} if the $d|V|$ coordinates of $\p(V)$ are algebraically independent over the rationals.
A graph $G$ is \defn{$d$-rigid} if $(G,\p)$ is rigid for a generic embedding $\p:V\to\RR^d$. It is standard that if $G$ is $d$-rigid, then it is $d'$-rigid for every $1\le d'\le d$. For notational convenience, we consider all graphs to be $0$-rigid.

In recent years, much work has been devoted to understanding rigidity properties of various models of random graphs, in various density regimes \cites{JSS07,KMT11c,The09c,KT13+,JT22,LNPR23,KLM25,KLM26+,PP24+}.
Pseudorandom graphs are, roughly speaking, families of graphs that share various key properties expected from random graphs, such as expansion properties, well-connectedness, and a relatively uniform distribution of edges.  These include many natural and well-studied families of random graphs, such as binomial random graphs, random regular graphs, and random Cayley graphs~\cites{AR94}. Moreover, there exist many families of pseudorandom graphs that are deterministic, yet successfully mimic typical key properties of truly random graphs (see, for example, \cites{LPS88,RVW02}). For more background on pseudorandom graphs, see \cites{KS06}.

It is then natural to study the rigidity of pseudorandom graphs. In addition to being, in a sense, a far-reaching extension of recent work on random graphs, one can see it as a further example of a combinatorial condition implying rigidity, continuing a quickly expanding line of work on the subject, including  Vill\'anyi's breakthrough results on vertex-connectivity conditions for rigidity and related further work \cites{Vil25,GJKV25,Gar26}, minimum degree conditions for rigidity \cites{KLM25,KLM26,JLV25+}, and partition-based sufficient conditions for rigidity \cites{LNPR25,KLM25,KLM26+}.

In this paper, we continue our work initiated in \cites{KLM25,KLM26+} on the rigidity of pseudorandom graphs, improving upon and extending earlier results, and giving essentially sharp conditions for the rigidity of expanders and pseudorandom graphs.  First, we consider $C$-expander graphs, as defined by Dragani\'c, Montgomery, Munh\'a Correia, Pokrovskiy, and Sudakov in their spectacular result on Hamiltonicity of pseudorandom graphs~\cites{DMMPS24+}.

\begin{definition}
\label{def:c_expander}
  For $C>0$, an $n$-vertex graph $G=(V,E)$ is a \defn{$C$-expander} if the following hold, where $N(A)$ denotes the external neighbourhood of $A$.
  \begin{enumerate}[label=(\roman*)]
    \item For every $A\subseteq V$ with $|A|< n/(2C)$, we have $|N(A)|\ge C|A|$.
    \item There is an edge between any two disjoint sets $A,B\subseteq V$ with $|A|,|B|\ge n/(2C)$.
  \end{enumerate}
\end{definition}

\begin{theorem}
\label{thm:c_expander}
  For every $C\ge 8$ and every integer $n\ge 9C$, every $n$-vertex $C$-expander is $\floor{C/8}$-rigid.
\end{theorem}

The dependence on $C$ in \cref{thm:c_expander} is best possible up to the value of the absolute constant $1/8$.
Indeed, there are $C$-expanders containing vertices of degree comparable to $C$, while every $d$-rigid graph with at least $d+1$ vertices has minimum degree at least $d$.
For a concrete example, let $C\ge 2$, let $n>2C$ be an integer, and let $G$ be obtained from a clique on $n-1$ vertices by adding a new vertex $v$ adjacent to exactly $\ceil{C}$ vertices of the clique.
Then $G$ is a $C$-expander: if $A=\{v\}$ then $|N(A)|=\ceil{C}\ge C$; if $0<|A|<n/(2C)$ and $A$ contains a clique vertex, then $|N(A)|\ge n-|A|-1\ge C|A|$; and any two disjoint sets of size at least $n/(2C)>1$ both contain clique vertices, which are adjacent.
Since $\delta(G)=\ceil{C}$, this $C$-expander is not $(\ceil{C}+1)$-rigid.

We next turn to the main result of the paper, which concerns spectral pseudorandomness.
One standard model for pseudorandom regular graphs is to require all non-trivial eigenvalues of the adjacency matrix to be small; see~\cites{KS06} for background.
\begin{definition}
\label{def:nrl_graph}
  An \defn{$(n,r,\lambda)$-graph} is an $r$-regular graph on $n$ vertices such that every eigenvalue of its adjacency matrix other than the largest eigenvalue $r$ has absolute value at most $\lambda$.
\end{definition}

The Expander Mixing Lemma (see \cref{thm:xml}) readily shows that every $(n,r,\lambda)$-graph with $r/\lambda\ge 8$ is an $r/(4\lambda)$-expander, and therefore \cref{thm:c_expander} implies that such a graph is $\floor{r/(32\lambda)}$-rigid whenever $r/\lambda\ge 32$ and $n\ge 9r/(4\lambda)$.
This should be compared with~\cite{KLM25}*{Theorem~1.7}, which, together with~\cite{KLM25}*{Corollary~1.4}, gives $d$-rigidity under the condition $r\ge\max\{9d\lambda,C_0d\log d\}$, for an absolute constant $C_0$; in the notation above, \cref{thm:c_expander} yields rigidity of order $r/\lambda$.
We go substantially beyond this estimate and prove a \emph{linear-in-the-degree} rigidity bound, assuming only that the non-trivial eigenvalues are at most a fixed small multiple of $r$.
This answers a question we raised in~\cite{KLM26+}.

\begin{mainthm}
\label{thm:main}
  There exist absolute constants $c_1,c_2>0$ such that every $(n,r,\lambda)$-graph $G$ with $\lambda\le c_1 r$ is $d$-rigid for $d=\floor{c_2r}$.
\end{mainthm}

The linear dependence on $r$ in the conclusion of \cref{thm:main} is best possible up to the value of the constant $c_2$.
Indeed, by the standard rigidity-matrix rank bound~\cite{AR78}, if an $n$-vertex graph with $n\ge d+1$ is $d$-rigid, then it has at least $dn-\binom{d+1}{2}$ edges.
For an $r$-regular graph, this gives $d\le (1/2+o(1))r$ when $r=o(n)$.
It would be interesting to decide whether this edge-counting obstruction is the only asymptotic obstruction for spectral expanders; in particular, whether every sequence of $(n,r,\lambda)$-graphs with $\lambda/r\to 0$ is $(1/2-o(1))r$-rigid.

We prove \cref{thm:main} by working in the more general setting of jumbled graphs.
For a graph $G=(V,E)$ and $A,B\subseteq V$, let $e(A,B)$ denote the number of ordered pairs $(u,v)\in A\times B$ such that $\{u,v\}\in E$.
\begin{definition}
\label{def:jumbled_graph}
Let $0<p<1$ and $\beta>0$. A graph $G=(V,E)$ on $n$ vertices is called \defn{$(p,\beta)$-jumbled} if, for all $A,B\subseteq V$,
  \[
    |e(A,B)- p |A| |B|| \le \beta \sqrt{|A| |B|}.
  \]
\end{definition}

\begin{mainthm}
\label{thm:main_jumbled}
  There exist absolute constants $c_1,c_2>0$ such that for every $n$-vertex $(p,\beta)$-jumbled graph $G=(V,E)$ with $\beta \le c_1 np$ there exists $S\subseteq V$ with $|S|\ge (1-2(\beta/(np))^2)n\ge (1-2c_1^2)n$, such that $G[S]$ is $d$-rigid for $d=\floor{c_2 np}$.
\end{mainthm}

\Cref{thm:main_jumbled} strengthens the corresponding jumbled-graph result from~\cite{KLM25}*{Proposition~4.4}, together with~\cite{KLM25}*{Corollary~1.4}.
In \cref{thm:main,thm:main_jumbled}, we may take $c_1=10^{-21}$ and $c_2=10^{-6}$. We made no serious attempt at optimizing these constants.

\subsection{Applications}

Recall that for a sequence $\{A_n\}_{n=1}^{\infty}$ of events in a probability space, we say that $A_n$ occurs \emph{with high probability}, or \whp{} for short, if $\pr(A_n)\to 1$ as $n\to\infty$.
Next, we present a few immediate applications of our main results.
Let $\cG_{n,r}$ denote the uniform distribution on $r$-regular graphs on vertex set $[n]$, where $1\le r<n$ and $nr$ is even.
Combining \cref{thm:main} with known spectral estimates for random regular graphs (see \cref{thm:random_regular_spectral} for details) gives the following consequence.

\begin{corollary}
\label{cor:random_regular}
  There exists an absolute constant $c>0$ such that the following holds.
  Let $3\le r=r(n)<n$ satisfy $nr$ even, and let $G\sim \cG_{n,r}$.
  Then $G$ is \whp{} $\floor{cr}$-rigid.
\end{corollary}

\cref{cor:random_regular} recovers, up to the value of the absolute constant, the random-regular graph result from~\cite{KLM26+}, while also extending it by allowing the degree $r$ to grow with $n$.

We next record an application of \cref{thm:c_expander} to graphs satisfying a purely global connectivity condition.
For a positive integer $s$, call a graph $G$ an \defn{$s$-connector} if every two disjoint $s$-element vertex sets are joined by an edge.

\begin{corollary}
\label{cor:connector}
  Let $s,n$ be positive integers with $n\ge 72s$, and let $G$ be an $n$-vertex $s$-connector.
  Then there exists $U\subseteq V(G)$ with $|U|>n-s$ such that $G[U]$ is $\floor{n/(72s)}$-rigid.
\end{corollary}

The bound on $|U|$ is best possible: the disjoint union of $K_{n-s+1}$ and $s-1$ isolated vertices is an $s$-connector.

As a final application, we consider $f$-connected graphs.
Let $f: \mathbb{N}\to \mathbb{R}$. We say that a graph $G=(V,E)$ is \defn{$f$-connected} if, for every $A,B\subsetneq V$ such that $A\cup B=V$ and $G$ has no edges between $A\setminus B$ and $B\setminus A$, the inequality
\[
    |A\cap B|\ge f(\min\{|A\setminus B|, |B\setminus A|\})
\]
holds. The notion of $f$-connectedness was introduced by Brandt, Broersma, Diestel, and Kriesell in \cite{BBDK06} as a property interpolating between vertex connectivity (corresponding to the case of $f$ being a constant function) and expansion (corresponding to the case of linearly growing $f$). Indeed, as shown in \cite{DMMPS24+}, it is easy to prove (see \cref{sec:applications} for details) that for every constant $C\ge 4$, if $f(k)=C k$ for every $k\in\mathbb{N}$, then every $f$-connected graph $G$ is a $C/2$-expander.
Therefore, as an immediate application of \cref{thm:c_expander}, we obtain the following.
\begin{corollary}\label{cor:f_connected}
    Let $C>16$ and $n\ge 9C$. Let $f:\mathbb{N}\to \mathbb{R}$ be defined by $f(k)=C k$ for all $k$. Then, every $n$-vertex $f$-connected graph is $\lfloor C/16\rfloor$-rigid.
\end{corollary}

It is easy to verify that the same example we presented above for showing the sharpness of \cref{thm:c_expander} implies that \cref{cor:f_connected} is tight up to the value of the constant $1/16$.
In fact, for every $f:\mathbb{N}\to\mathbb{N}$ with $n\ge f(1)+2$, the graph obtained from a clique on $n-1$ vertices by adding a vertex adjacent to exactly $f(1)$ clique vertices is $f$-connected, but is not $(f(1)+1)$-rigid.
On the other hand, it is easy to show that, for monotone non-decreasing $f$, every $n$-vertex $f$-connected graph is an $s$-connector whenever $f(s)>n-2s$, and in particular for $s=\min \{k\in \mathbb{N}:\,  f(k)\ge n\}$. Hence, by \cref{cor:connector}, assuming $n\ge 72s$, such a graph has an induced subgraph on more than $n-s$ vertices which is $\lfloor n/(72 s)\rfloor$-rigid.  It may be of interest to study the rigidity of $f$-connected graphs (as a function of $f(1)$) for more general values of $f$, potentially interpolating between our results here and Vill\'anyi's vertex-connectivity condition \cite{Vil25}.

\begin{remark}
    A graph $G=(V,E)$ is called \emph{generically globally rigid} in $\mathbb{R}^d$ if, for a generic embedding $\p:V\to \mathbb{R}^d$, every embedding $\textbf{q}:V\to \mathbb{R}^d$ satisfying $\|\textbf{q}(u)-\textbf{q}(v)\|=\|\p(u)-\p(v)\|$ for all $\{u,v\}\in E$ satisfies in fact $\|\textbf{q}(u)-\textbf{q}(v)\|=\|\p(u)-\p(v)\|$ for all $u,v\in V$.
    It was shown by Jord\'an \cite{Jor17}, relying on a theorem of Tanigawa \cite{Tan15}, that every $(d+1)$-rigid graph is generically globally rigid in $\mathbb{R}^d$. Using this fact, one may immediately strengthen the conclusions in \cref{thm:c_expander,thm:main} from rigidity to generic global rigidity (changing the bound $\lfloor C/8\rfloor $ to $\lfloor C/8\rfloor-1$ in the case of \cref{thm:c_expander}).
\end{remark}

\medskip

The paper is organized as follows.
In \cref{sec:preliminaries} we collect preliminaries on pseudorandom graphs and rigidity.
In \cref{sec:c_expanders} we prove \cref{thm:c_expander}.
In \cref{sec:jumbled_properties} we develop random-partition estimates for jumbled graphs, which are used in the following section.
In \cref{sec:main_proof} we prove \cref{thm:main_jumbled,thm:main}.
Finally, in \cref{sec:applications} we prove \cref{cor:random_regular,cor:connector,cor:f_connected}. Throughout the paper, all logarithms are in the natural base.

\section{Preliminaries}
\label{sec:preliminaries}

For a graph $G=(V,E)$ and sets $A,B\subseteq V$, write $e(A,B)=\left|\left\{(u,v)\in A\times B:\{u,v\}\in E\right\}\right|$.
For disjoint sets $X,Y\subseteq V$, write $E_G(X,Y)=\{e\in E:\, |e\cap X|=|e\cap Y|=1\}$; we suppress the subscript when $G$ is clear from context.
For an integer $m$, let $[m]=\{1,2,\ldots,m\}$ and $\binom{[m]}{2}=\{\{i,j\}:\, 1\le i<j\le m\}$.
For a finite set $V$, we say that $V_1,\ldots,V_m$ is a \defn{partition} of $V$ if $V=V_1\cup \cdots \cup V_m$, $V_i\cap V_j=\es$ for all $1\le i<j\le m$, and $V_i\ne \es$ for all $1\le i\le m$.
If, in addition, the sizes $|V_1|,\ldots,|V_m|$ differ by at most $1$, we call this partition an \defn{equipartition}.
For convenience, whenever we write an equipartition $V_1,\ldots,V_m$, we assume $|V_1|\ge\cdots\ge |V_m|$.
Equivalently, if $h=|V|-m\floor{|V|/m}$, then $|V_i|=\ceil{|V|/m}$ for $1\le i\le h$, and $|V_i|=\floor{|V|/m}$ for $h<i\le m$.
A \defn{uniformly random labelled equipartition} is a partition chosen uniformly among all labelled partitions with these prescribed part sizes.
Let $G=(V,E)$ be a graph, let $V_1,\ldots,V_m$ be a partition of $V$, and let $K$ be a positive integer.
A pair $\{i,j\}\in\binom{[m]}{2}$ is \defn{$K$-good} if, for every $X\subseteq V_i$ and $Y\subseteq V_j$ with $|X|=|Y|=K$, we have $E_G(X,Y)\ne\es$.
Otherwise, $\{i,j\}$ is \defn{$K$-bad}.
The \defn{$K$-reduced graph} associated with $G$ and the partition $V_1,\ldots,V_m$ is the graph $R=([m],E')$ with
\[
  E'=\left\{\{i,j\}\in\binom{[m]}{2}:\, \{i,j\}\text{ is $K$-good}\right\}.
\]

We will use the following simple observation to clean a sparse graph by discarding high-degree vertices.

\begin{lemma}
\label{lem:cleanup_reduced}
  Let $p\in[0,1/9]$.
  Let $G=(V,E)$ be an $n$-vertex graph with $|E|\le pn^2$.
  Then there exists $S\subseteq V$ with
  $|S|\ge (1-2\sqrt{p})n$ satisfying
  $\Delta(G[S])\le 3\sqrt{p}|S|$.
\end{lemma}

\begin{proof}
  If $p=0$, then $E=\es$, and we may take $S=V$.
  Assume then that $p>0$.
  Let
  \[
    S=\left\{v\in V:\deg_G(v)\le \sqrt{p}\,n\right\}.
  \]
  Then $|V\sm S|\cdot \sqrt{p}n \le 2|E|\le 2pn^2$,
  so $|S|\ge (1-2\sqrt{p})n$.
  For every $v\in S$,
  we have $\deg_{G[S]}(v)\le \deg_G(v)\le \sqrt{p}\,n$.
  Since $\sqrt{p}\le 1/3$ and $|S|\ge (1-2\sqrt{p})n$,
  we have $\Delta(G[S])\le \sqrt{p}\,n\le 3\sqrt{p}|S|$.
\end{proof}

\subsection{\texorpdfstring{$(n,r,\lambda)$-graphs}{(n,r,lambda)-graphs}}

The Expander Mixing Lemma goes back to Alon and Chung~\cite{AC88}.
We use the version for two arbitrary vertex subsets stated in~\cite{HLW06}*{Lemma~2.5}.
\begin{theorem}[Expander Mixing Lemma]
\label{thm:xml}
  Let $G=(V,E)$ be an $(n,r,\lambda)$-graph.
  Then, for all $A,B\subseteq V$,
  \[
    \left|e(A,B)-\frac{r}{n}|A||B|\right|
    \le \lambda\sqrt{|A||B|}.
  \]
\end{theorem}

Given a graph $G=(V,E)$ and a vertex set $U\subseteq V$, denote by $\partial U=E(U,V\sm U)$ the edge boundary of $U$.
Let
\[
  \iso(G)=\min_{\substack{U\subseteq V,\\ 1\le |U|\le |V|/2}}\frac{|\partial U|}{|U|}
\]
be the \defn{isoperimetric number} of $G$.
We use the following standard spectral expansion bound due to Alon and Milman~\cite{AM85}; see also~\cite{HLW06}*{Theorem~4.11}.

\begin{lemma}
\label{lem:spectral_iso}
  Let $G=(V,E)$ be an $(n,r,\lambda)$-graph.
  Then $\iso(G)\ge (r-\lambda)/2$.
\end{lemma}

\subsection{Rigidity}

\begin{lemma}[$0$-extension; e.g.,~\cite{TW85}]
\label{lem:0_extension}
  Let $G=(V,E)$ be a $d$-rigid graph, and let $G'$ be obtained from $G$ by adding a new vertex $v$ adjacent to at least $d$ vertices from $V$.
  Then $G'$ is $d$-rigid.
\end{lemma}

The following absorption lemma, which follows easily from \cref{lem:0_extension}, is the case $k=1$ of~\cite{KLM26+}*{Lemma~3.3}.
\begin{lemma}[Absorption]
\label{lem:absorption}
  Let $d\ge 1$ and let $G=(V,E)$ be a graph with $\iso(G)\ge d$.
  Let $W\subseteq V$ be such that $|W|\ge |V|/2$ and $G[W]$ is $d$-rigid.
  Then $G$ is $d$-rigid.
\end{lemma}

Finally, we need the following result from \cite{KLM26+}, which is a main ingredient in our proofs of \cref{thm:c_expander,thm:main}.

\begin{theorem}[{\cite{KLM26+}*{Theorem~1.5}}]
\label{thm:connector_partition}
  Let $G=(V,E)$ be a graph, let $K,m$ be positive integers, and let $1/2<\eta\le 1$.
  Let $V_1,\ldots,V_m$ be a partition of $V$ satisfying $|V_i|>7K-3$ for all $1\le i\le m$.
  Let $G_0$ be the $K$-reduced graph associated with $G$ and this partition.
  Assume that $\delta(G_0)\ge \eta m-1$.
  Then there exists $W\subseteq V$ such that $G[W]$ is $d$-rigid for $d=\floor{(\eta-1/2)m}$, and $|V_i\sm W|\le 4K$ for all $1\le i\le m$.
\end{theorem}

\section{\texorpdfstring{$C$-expanders}{C-expanders}}
\label{sec:c_expanders}

\begin{proof}[Proof of \cref{thm:c_expander}]
  Let $G=(V,E)$ be an $n$-vertex $C$-expander, where $C\ge 8$ and $n\ge 9C$.
  Put
  \[
    d=\floor{C/8},\qquad
    K=\ceil{n/(2C)},\qquad
    m=\floor{C/4}.
  \]
  Then $\floor{m/2}=\floor{C/8}=d$.
  Let $V_1,\ldots,V_m$ be an arbitrary equipartition of $V$.
  Set $x=n/(2C)$ and $k=\ceil{x}=K$, so that $x\ge 9/2$ and $k\ge 5$.
  Since $m\le C/4$, every part satisfies $|V_i|\ge\floor{8x}$.
  If $k=5$, then $\floor{8x}\ge 36>7k-3$.
  If $k\ge 6$, then $x>k-1$, and hence $\floor{8x}\ge 8k-7>7k-3$.
  Hence every part satisfies $|V_i|>7K-3$.

  Let $G_0$ be the $K$-reduced graph associated with $G$ and the partition $V_1,\ldots,V_m$.
  Since $K\ge n/(2C)$, the second condition in \cref{def:c_expander} gives $G_0=K_m$.
  Applying \cref{thm:connector_partition} with $\eta=1$, we obtain a set $W_0\subseteq V$ such that $G[W_0]$ is $\floor{m/2}$-rigid, and
  \[
    |V_i\sm W_0|\le 4K
    \quad\text{for all }1\le i\le m.
  \]
  In particular,
  \begin{equation}
  \label{eq:c_expander_w0_size}
    |W_0|
    \ge n-4mK
    \ge n\left(1-\frac{2m}{C}\right)-4m
    \ge \frac{n}{2}-C.
  \end{equation}
  Since $\floor{m/2}\ge d$, the graph $G[W_0]$ is $d$-rigid.

  Extend $W_0$ to an inclusion-maximal set $W\subseteq V$ such that $G[W]$ is $d$-rigid, and put $U=V\sm W$.
  We claim that $U=\es$.
  By maximality and \cref{lem:0_extension}, every vertex of $U$ has at most $d-1$ neighbours in $W$.
  Consequently,
  \begin{equation}
  \label{eq:c_expander_uw_edges}
    |E_G(U,W)|\le (d-1)|U|.
  \end{equation}

  Suppose first that $0<|U|<K$.
  Then $|U|<n/(2C)$, so \cref{def:c_expander} gives $|N(U)|\ge C|U|$.
  Since $U=V\sm W$, we have $N(U)\subseteq W$.
  Therefore,
  \[
    C|U|\le |N(U)|\le |E_G(U,W)|\le (d-1)|U|<C|U|,
  \]
  a contradiction.
  Hence either $U=\es$ or $|U|\ge K$.

  Suppose, for contradiction, that $|U|\ge K$.
  For every $K$-element set $A\subseteq U$, we have $|W\sm N(A)|<K$.
  Indeed, otherwise there would be a set $B\subseteq W\sm N(A)$ with $|B|=K$, and then $A$ and $B$ would be disjoint sets of size at least $n/(2C)$ with no edge between them, contradicting \cref{def:c_expander}.
  Thus every such set $A$ sends more than $|W|-K$ edges to $W$.
  Since $W_0\subseteq W$, \eqref{eq:c_expander_w0_size} gives $|W|\ge n/2-C$.
  Also, since $d\le C/8$, $K\le n/(2C)+1$, and $n\ge 3C$,
  \[
    dK\le \frac{n}{16}+\frac{C}{8}
    \le \frac{n}{16}+\frac{n}{24}
    =\frac{5n}{48}
    <\frac{n}{6}
    \le \frac{n}{2}-C
    \le |W|.
  \]
  By averaging over all $K$-element subsets of $U$, and using $dK<|W|$, we obtain
  \[
    |E_G(U,W)|
    >\frac{(|W|-K)|U|}{K}
    >(d-1)|U|,
  \]
  contradicting \eqref{eq:c_expander_uw_edges}.
  Therefore $U=\es$, and so $G=G[W]$ is $d$-rigid.
\end{proof}

\section{Pseudorandom properties of jumbled graphs}
\label{sec:jumbled_properties}

Fix $p\in(0,1)$. For an $n$-vertex graph $G=(V,E)$ and a set $U\subseteq V$, write
\[
  B(U)=\left\{v\in V:\, |N(v)\cap U|<\frac{p|U|}{2}\right\},
\]
the set of \defn{bad vertices with respect to $U$}.

The following simple lemma is a mild variation of~\cite{Kri16}*{Corollary~2.2}.
\begin{lemma}
\label{lem:large_set_degrees}
  Let $G=(V,E)$ be an $n$-vertex $(p,\beta)$-jumbled graph.
  Write $\eps=\beta/(np)$.
  Let $\alpha\in(0,1)$, and let $U\subseteq V$ satisfy $|U|\ge \alpha n$.
  Then
  \[
    |B(U)|\le (4\eps^2/\alpha)n.
  \]
\end{lemma}

\begin{proof}
  Let $B=B(U)$, and write $b=|B|$.
  If $b=0$, there is nothing to prove. Hence assume $b\ge 1$. By the definition of $B(U)$,
  \[
    e(B,U)=\sum_{v\in B}|N(v)\cap U|
    <\frac{p b|U|}{2}.
  \]
  On the other hand, by the definition of $(p,\beta)$-jumbledness,
  \[
    e(B,U)\ge  p b|U|-\beta \sqrt{b|U|}.
  \]
  Therefore,
  \[
    p b |U|/2<\beta\sqrt{b|U|}.
  \]
  Squaring gives
  \[
    b<\frac{4\beta^2 }{p^2|U|}
    \le (4\eps^2/\alpha)n,
  \]
  as required.
\end{proof}

For an $n$-vertex graph $G=(V,E)$, a number $\alpha\in(0,1)$, and a set $X\subseteq V$, we say that $X$ is \defn{$\alpha$-poor} if its closed neighbourhood misses at least an $\alpha$-fraction of the vertices, namely if $|V\sm (X\cup N(X))|\ge \alpha n$.

The proof of the next lemma is based on the same ordered-exposure counting strategy as~\cite{DK23}*{Lemma~2.4}.
\begin{lemma}
\label{lem:poor_sets_general}
  Let $G=(V,E)$ be an $n$-vertex $(p,\beta)$-jumbled graph.
  Write $\eps=\beta/(np)$.
  Let $\alpha\in(0,1)$, and let $K\ge 3\log(1/\alpha)/p$ be an integer.
  Put $\mu=4\eps^2/\alpha$.
  Then the number of $\alpha$-poor $K$-sets is at most
  \[
    \left(\frac{2e\mu^{1/4} n}{K}\right)^K.
  \]
\end{lemma}

\begin{proof}
  If $K>n$, then there are no $K$-sets, so we may assume that $K\le n$.
  If $\mu\ge 1$, then the number of $\alpha$-poor $K$-sets is at most
  \[
    \binom{n}{K}
    \le \left(\frac{en}{K}\right)^K
    \le \left(\frac{2e\mu^{1/4}n}{K}\right)^K,
  \]
  so the result follows.
  Hence assume that $\mu<1$.
  Take an ordered $K$-tuple of distinct vertices $(x_1,\dots,x_K)$.
  For $1\le t\le K$, let $S_t=\{x_1,\dots,x_t\}$ and $U_t=V\sm (S_t\cup N(S_t))$. In addition, let $U_0=V$.

  Suppose that $S_K$ is $\alpha$-poor.
  Then $|U_K|\ge \alpha n$.
  Since $U_0\supseteq U_1\supseteq \cdots \supseteq U_K$, we have $|U_{t-1}|\ge \alpha n$ for every $1\le t\le K$.
  For $1\le t\le K$, call time $t$ \defn{exceptional} if $x_t\in B(U_{t-1})$. By \cref{lem:large_set_degrees}, for fixed $x_1,\ldots,x_{t-1}$, there are at most $\mu n$ possible values of $x_t$ that make time $t$ exceptional.

  We claim that every $\alpha$-poor $K$-tuple (that is, every ordering of every $\alpha$-poor $K$-set) has at least $K/4$ exceptional times.
  Indeed, assume for contradiction that fewer than $K/4$ times are exceptional.
  Then at least $3K/4$ times are non-exceptional.
  At every non-exceptional time $t$,
  \[
    |N(x_t)\cap U_{t-1}|\ge \frac{p|U_{t-1}|}{2}.
  \]
  Hence
  \[
    |U_t|
    \le |U_{t-1}|-|N(x_t)\cap U_{t-1}|
    \le \left(1-p/2\right)|U_{t-1}|.
  \]
  At exceptional times, we still have $U_t\subseteq U_{t-1}$.
  Therefore,
  \[
    |U_K|
    \le n\left(1-p/2\right)^{3K/4}
    \le n\exp\left(-\frac{3p K}{8}\right)
    \le n\exp\left(-\frac{9}{8}\log(1/\alpha)\right)
    <\alpha n,
  \]
  a contradiction.

  Next, we count $\alpha$-poor $K$-tuples in the following way. Let $\ell=\ceil{K/4}$. First, we choose $\ell$ times in which the $K$-tuple is required to be exceptional. Then, for each chosen time $t$, once $x_1,\ldots,x_{t-1}$ have been fixed, there are at most $\mu n$ choices for $x_t$. For every other time $t$, we use the trivial bound of $n$ choices for $x_t$. Note that every $\alpha$-poor $K$-tuple is counted at least once in this way. Hence, the number of $\alpha$-poor $K$-tuples is at most
  \[
    \binom{K}{\ell}(\mu n)^\ell n^{K-\ell}
    =\binom{K}{\ell}\mu^\ell n^K
    \le 2^K\mu^{K/4}n^K,
  \]
  where the last inequality uses $\mu<1$.
  Each $\alpha$-poor $K$-set has $K!$ orderings.
  Since $K!\ge (K/e)^K$,
  the number of $\alpha$-poor $K$-sets is at most
  \[
    \frac{2^K\mu^{K/4}n^K}{K!}
    \le \left(\frac{2e\mu^{1/4} n}{K}\right)^K.\qedhere
  \]
\end{proof}

\begin{lemma}
\label{lem:random_parts}
  Let $G=(V,E)$ be an $n$-vertex $(p,\beta)$-jumbled graph.
  Write $\eps=\beta/(np)$. Let $\alpha\in(0,1)$ and $C>1$.
  Let $n\ge m\ge np/(C-1)$ and $K\in[3\log(1/\alpha)/p,n/4]$ be integers.
  Put $\mu=4\eps^2/\alpha$ and $\theta=C^2/\log^2(1/\alpha)$, and assume that $\mu\le 1$.
  Let $V=V_1\cup\cdots\cup V_m$ be a uniformly random labelled equipartition.
  Then, for every fixed $\{i,j\}\in\binom{[m]}{2}$,
  \[
    \pr(\{i,j\}\text{ is $K$-bad})
    \le
    \left(\frac{8e^2}{9}\theta\mu^{1/4}\right)^K
    +
    \left(\frac{4e^2}{9}\theta\alpha\right)^K.
  \]
\end{lemma}

\begin{proof}
  Fix $i\ne j$, and let $n_s=|V_s|$ for $s\in[m]$.
  Since $p<1$, for every $s\in[m]$ we have
  \[
    n_s\le \ceil{\frac{n}{m}}
    \le \frac{n}{m}+1
    \le \frac{(C-1)}{p}+\frac{1}{p}
    =\frac{C}{p}.
  \]
  Now fix two disjoint $K$-sets $X,Y\subseteq V$.
  Viewing the equipartition as a uniformly random assignment to labelled slots, the $2K$ vertices of $X\cup Y$ have $(n)_{2K}$ equally likely placements, of which $(n_i)_K(n_j)_K$ place $X$ in $V_i$ and $Y$ in $V_j$.
  Therefore, using $K\le n/4$, we have
  \begin{equation}
  \label{eq:place_pair}
    \pr(X\subseteq V_i,\ Y\subseteq V_j)
    =\frac{(n_i)_K(n_j)_K}{(n)_{2K}}
    \le \frac{n_i^Kn_j^K}{(n/2)^{2K}}
    =\left(\frac{2n_i}{n}\right)^K
     \left(\frac{2n_j}{n}\right)^K
    \le \left(\frac{2C}{np}\right)^{2K},
  \end{equation}
  where, for non-negative integers $a,b$, $(a)_b=a(a-1)\cdots(a-b+1)$.
  Let $\cP$ be the family of $\alpha$-poor $K$-sets.
  By \cref{lem:poor_sets_general},
  \[
    |\cP|\le \left(\frac{2e\mu^{1/4} n}{K}\right)^K.
  \]
  Let $N_{\es}$ be the set of ordered pairs $(X,Y)$ of disjoint $K$-subsets of $V$ with $E_G(X,Y)=\es$.
  Note that, for a $K$-set $X\notin \cP$, if $(X,Y)\in N_{\es}$ for some $K$-set $Y$, then $Y\subseteq V\sm (X\cup N(X))$. As $|V\sm (X\cup N(X))|< \alpha n$, we obtain that there are at most $\binom{\lfloor \alpha n\rfloor}{K}$ $K$-sets $Y$ such that $(X,Y)\in N_{\es}$.
  Hence, using $\binom{n}{K}\le (en/K)^K$, we obtain
  \[
  \begin{aligned}
    |N_{\es}|
    &\le |\cP|\binom{n}{K}
       +\binom{n}{K}\binom{\floor{\alpha n}}{K} \\
    &\le
    \left(2e^2\mu^{1/4}\left(\frac{n}{K}\right)^2\right)^K
    +
    \left(e^2\alpha\left(\frac{n}{K}\right)^2\right)^K.
  \end{aligned}
  \]
  Also, by the definition of $\theta$ and the lower bound on $K$, we have $(C/(pK))^2\le \theta/9$.
  By the union bound and \eqref{eq:place_pair},
  \[
  \begin{aligned}
    \pr(\{i,j\}\text{ is $K$-bad})
    &\le |N_{\es}|\left(\frac{2C}{np}\right)^{2K}
    \le
    \left(
      2e^2\mu^{1/4}
      \left(\frac{2C}{p K}\right)^2
    \right)^K
    +
    \left(
      e^2\alpha
      \left(\frac{2C }{p K}\right)^2
    \right)^K \\
    &\le
    \left(\frac{8e^2}{9}\theta\mu^{1/4}\right)^K
    +
    \left(\frac{4e^2}{9}\theta\alpha\right)^K.
  \end{aligned}
  \]
  This is the required bound.
\end{proof}

\section{\texorpdfstring{Proof of \cref{thm:main_jumbled,thm:main}}{Proof of the main theorems}}
\label{sec:main_proof}

\begin{proof}[Proof of \cref{thm:main_jumbled}]

Let $c_1=10^{-21}$ and $c_2=10^{-6}$. Let $G$ be an $n$-vertex $(p,\beta)$-jumbled graph with $\beta\le c_1 np$. Let $d=\lfloor c_2 np \rfloor$. If $d=0$, there is nothing to prove. Hence, we assume $d\ge1$. That is, $np\ge c_2^{-1}=10^6$. Let
\[
    m=\lfloor np\cdot 10^{-4}\rfloor
\]
and
\[
    K= \lceil 100/p \rceil.
\]
Note that $100\le m\le n/10000$.

The proof has four steps.
  First, we sample a uniformly random labelled equipartition of $V$ into $m$ parts, and consider its bad-pair graph (that is, the complement of its $K$-reduced graph).
  \Cref{lem:random_parts} says that, for each fixed pair of parts, the probability that this pair is bad is small; hence some equipartition has a sparse bad-pair graph.
  Then, \cref{lem:cleanup_reduced} lets us delete a small set of indices and obtain an induced subgraph $R_0$ of the corresponding $K$-reduced graph with minimum degree slightly larger than half its order.
  Second, we apply \cref{thm:connector_partition} to the subgraph of $G$ induced by the parts whose indices survive this deletion, obtaining a $d$-rigid subgraph $G[W]$.
  Third, we show that this rigid subgraph contains more than $90\%$ of the vertices of $G$: the cleanup step deletes few parts, and the set $W$ obtained from \cref{thm:connector_partition} contains almost all vertices from each surviving part.
  Finally, we use \cref{lem:0_extension} to extend $d$-rigidity from $G[W]$ to a set $S$ of size $|S|\ge (1-2(\beta/(np))^2)n$.

 \paragraph{Step 1: A random partition.}
  For an equipartition $\rho=(U_1,\ldots,U_m)$ of $V$, let $B_\rho$ be the graph on $[m]$ whose edges are the pairs $\{i,j\}$ that are $K$-bad with respect to $U_i,U_j$.
  Let $R_\rho$ be the $K$-reduced graph associated with $G$ and $\rho$. Thus $R_\rho$ is the complement of $B_\rho$.

  Let $\pi$ be a uniformly random labelled equipartition of $V$ into $m$ parts. For \cref{lem:random_parts}, take $C=3\cdot 10^4$ and $\alpha=10^{-14}$. Then, we have $m\ge np\cdot 10^{-4}-1\ge np/(C-1)$. Also,
  \[
      3\log{(1/\alpha)}/p \le 100/p\le K\le 100/p+1 \le n/10^4+1\le n/4.
  \]
  Write $\eps=\beta/(np)\le c_1$, and let $\mu= 4 \eps^2/\alpha$. Note that $\mu\le 4 c_1^2/\alpha =4 \cdot 10^{-28}<1$.
  Using
  \[
    \theta=\frac{C^2}{\log^2(1/\alpha)}\le 9 \cdot 10^5
  \]
  we obtain, by \cref{lem:random_parts}, that for every fixed pair $\{i,j\}\in\binom{[m]}{2}$,
  \[
    \pr(\{i,j\}\in E(B_\pi))
    \le
     \left(\frac{8e^2}{9}\theta\mu^{1/4}\right)^K
    +
    \left(\frac{4e^2}{9}\theta\alpha\right)^K
    \le 10^{-7}.
  \]
  Thus, setting $\gamma=10^{-7}$, we obtain
  \[
    \E |E(B_\pi)|\le \gamma \binom{m}{2}< \gamma m^2.
  \]
  Hence there exists an equipartition $\sigma=(V_1,\ldots,V_m)$ for which $|E(B_\sigma)|\le \gamma m^2$.
  Fix such a $\sigma$.
  By \cref{lem:cleanup_reduced}, there exists $I\subseteq[m]$ such that, setting $R_0=R_{\sigma}[I]$ and $m'=|I|$,
  \begin{equation}
  \label{eq:many_indices}
    m'\ge (1-2\sqrt{\gamma})m
    \quad\text{and}\quad
    \Delta(B_\sigma[I])\le 3\sqrt{\gamma}\,m'.
  \end{equation}
  Since $R_\sigma$ is the complement of $B_\sigma$, we have
  \[
    \delta(R_0)
    \ge m'-1-\Delta(B_\sigma[I])
    \ge (1-3\sqrt{\gamma})m'-1 \ge 0.999 m'-1.
  \]

  \paragraph{Step 2: Application of \cref{thm:connector_partition}.}
  Since the partition $\sigma$ is an equipartition and $m\le np \cdot 10^{-4}$, we have $|V_i|\ge \floor{n/m}\ge 10^4/p-1> 7K-3$.
  Write
  \[
    V_I=\bigcup_{i\in I}V_i.
  \]
  Choose a bijection $\phi:[m']\to I$, and write $V'_a=V_{\phi(a)}$ for $a\in[m']$.
  The sets $V'_1,\ldots,V'_{m'}$ form a partition of $V_I$, satisfying $|V'_a|>7K-3$ for every $a\in[m']$.
  The $K$-reduced graph associated with $G[V_I]$ and this partition is isomorphic to $R_0$.
  Taking $\eta=0.999$ and applying \cref{thm:connector_partition}, we obtain a set $W\subseteq V_I$ such that $G[W]$ is $d_0$-rigid, where $d_0=\floor{0.499m'}$.
  Since $\gamma=10^{-7}$, $m'\ge (1-2\sqrt{\gamma})m$, $m\ge 10^{-4}np-1$, and $np\ge 10^6$, we have $d_0\ge d$.
  Thus $G[W]$ is $d$-rigid.
  Moreover,
  \begin{equation}
  \label{eq:connector_loss}
    |V_i\sm W|\le 4K
    \quad\text{for every }i\in I.
  \end{equation}

  \paragraph{Step 3: The rigid set is large.}
  We next show that $|W|>0.9 n$.
  By \eqref{eq:many_indices}, the discarded index set $[m]\sm I$ has size at most $2\sqrt{\gamma}m$.
  Since each part has size at most $n/m+1$,
  \begin{equation}
  \label{eq:discarded_vertices}
    |V\sm V_I|
    \le 2\sqrt{\gamma}m\left(\frac{n}{m}+1\right)
    =2\sqrt{\gamma}(n+m)
    \le 2\sqrt{\gamma}(1+10^{-4})n
    <10^{-3}n.
  \end{equation}
  Also, by \eqref{eq:connector_loss}, and using $m'\le m$, $K\le 100/p+1\le 200/p$ and $m\le np\cdot 10^{-4}$, we get
  \begin{equation}
  \label{eq:remaining_vertices}
    \sum_{i\in I}|V_i\sm W|
    \le 4Km
    \le 4(200/p) np \cdot 10^{-4} = 0.08n.
  \end{equation}
  Since $W\subseteq V_I$, we have $|W|=n-|V\sm V_I|-\sum_{i\in I}|V_i\sm W|$.
  Combining this identity with \eqref{eq:discarded_vertices} and \eqref{eq:remaining_vertices}, we get $|W|\ge 0.9 n$.

  \paragraph{Step 4: Absorption.} Let $S\supseteq W$ be a maximal set such that $G[S]$ is $d$-rigid. Then $|S|\ge |W|\ge 0.9n$. Put $U=V\sm S$. If $U=\es$, there is nothing to prove. Otherwise, by the maximality of $S$ and the $0$-extension property (\cref{lem:0_extension}), each vertex in $U$ has less than $d$ neighbours in $S$. Hence, $e(U,S)/|U|<d$. On the other hand, since $G$ is a $(p,\beta)$-jumbled graph, we have
  \[
    10^{-6}np\ge d>\frac{e(U,S)}{|U|}\ge p|S|-\beta\sqrt{\frac{|S|}{|U|}}\ge 0.9np-\beta\sqrt{\frac{n}{|U|}}.
  \]
  We obtain
  \[
    |U|\le \left(\frac{\beta}{(0.9-10^{-6})np}\right)^2 n\le 2 (\beta/(np))^2 n.
  \]
  So $|S|\ge (1-2(\beta/(np))^2)n$, as wanted.
\end{proof}

The proof of \cref{thm:main} follows as a simple corollary of \cref{thm:main_jumbled}.

\begin{proof}[Proof of \cref{thm:main}]
We may assume that $r\ge1$, as otherwise the statement is trivial.
Let $c_1=10^{-21}$ and $c_2=10^{-6}$, as in the proof of \cref{thm:main_jumbled}.
Let $G=(V,E)$ be an $(n,r,\lambda)$-graph with $\lambda\le c_1 r$, and put $p=r/n$ and $\beta=\lambda$.
By \cref{thm:xml}, $G$ is $(p,\beta)$-jumbled with $\beta\le c_1np$.
Thus, by \cref{thm:main_jumbled}, there is $S\subseteq V$ such that $|S|\ge (1-2c_1^2)n\ge n/2$ and $G[S]$ is $d$-rigid for $d=\floor{c_2np}= \floor{c_2 r}$.
If $d=0$, there is nothing to prove.
By \cref{lem:spectral_iso},
\[
  \iso(G)\ge (r-\lambda)/2 \ge (1-c_1)r/2\ge d.
\]
Therefore, by \cref{lem:absorption}, $G$ is $d$-rigid, as wanted.

\end{proof}

\section{Applications}
\label{sec:applications}

\subsection{Random regular graphs}

We use the following spectral estimate for random regular graphs.
Here $\lambda(G)$ denotes the maximum absolute value of the non-trivial adjacency eigenvalues of $G$.
\begin{theorem}
\label{thm:random_regular_spectral}
  Let $1\le r=r(n)<n$ satisfy $nr$ even, and let $G\sim\cG_{n,r}$.
  Then $\lambda(G)=O(\sqrt r)$ \whp{}.
\end{theorem}
For an asymptotically sharp estimate in the range $1\ll r\le \gamma n$, for some absolute constant $\gamma>0$, see Sarid~\cites{Sar23}.
\begin{proof}
For $1\le r\le n/2$, the estimate is recorded after Theorem~A of Tikhomirov--Youssef~\cites{TY19}, where it follows by combining their theorem with the earlier sparse estimates cited there.
For $r>n/2$, put $s=n-1-r<n/2$.
If $s=0$, then $G=K_n$ and $\lambda(G)=1$.
Otherwise, let $H$ be the complement of $G$, so that $H\sim\cG_{n,s}$, and let $x$ be an eigenvector of $G$ with eigenvalue $\mu\ne r$.
Since $G$ is regular, $x$ is orthogonal to the all-ones vector.
Writing $A_G,A_H$ for the adjacency matrices of $G,H$ and $J$ for the all-ones matrix, we have $A_H=J-I-A_G$, and hence
\[
  A_Hx=(J-I-A_G)x=(-1-\mu)x.
\]
Thus $x$ is an eigenvector of $H$ with eigenvalue $-1-\mu$.
Therefore $\lambda(G)\le 1+\lambda(H)$.
Applying the case $s\le n/2$ to $H$ gives $\lambda(H)=O(\sqrt s)$ \whp{}, and since $s<r$, we get $\lambda(G)=O(\sqrt r)$ \whp{}.
\end{proof}

\begin{proof}[Proof of \cref{cor:random_regular}]
  Let $c_1,c_2>0$ be the constants from \cref{thm:main}.
  By \cref{thm:random_regular_spectral}, there is an absolute constant $C_0$ such that $\lambda(G)\le C_0\sqrt r$ \whp{}.
  Choose $R$ so large that $C_0/\sqrt r\le c_1$ for every $r\ge R$.
  Choose $c>0$ such that $c\le c_2$ and $cR<1$.
  If $r<R$, then $\floor{cr}=0$, and the assertion is vacuous.
  If $r\ge R$, then \cref{thm:random_regular_spectral} gives $\lambda(G)\le c_1r$ \whp{}, so \cref{thm:main} gives that $G$ is \whp{} $\floor{c_2r}$-rigid, and hence also $\floor{cr}$-rigid.
\end{proof}

\subsection{Further applications}

\begin{proof}[Proof of \cref{cor:connector}]
  Put $C=n/(9s)$.
  Starting with $V(G)$, repeatedly delete a nonempty set $A$ of at most $s$ vertices whose external neighbourhood in the current graph has size less than $C|A|$.
  Fewer than $s$ vertices are deleted: otherwise, when the deleted set $D$ first has size at least $s$, we have $s\le |D|<2s$, and summing the corresponding inequalities gives $|N_G(D)|<C|D|<2n/9$.
  But the $s$-connector property gives $|N_G(D)|>n-|D|-s>n-3s>2n/9$, a contradiction.

  Let $U$ be the set left by the process, and write $H=G[U]$ and $N=|U|$.
  Thus $|U|>n-s$, and $|N_H(A)|\ge C|A|$ whenever $|A|\le s$.
  If $s<|A|<N/(2C)$, then the $s$-connector property gives $|N_H(A)|>N-|A|-s>N/2>C|A|$; here we used $N>n-s\ge71s$ and $N/(2C)\le n/(2C)=9s/2$.
  Moreover, $N/(2C)>4s>s$, so the $s$-connector property also gives an edge between every two disjoint sets of size at least $N/(2C)$.
  Hence $H$ is a $C$-expander.
  Finally, $C\ge8$ and $N\ge n-s+1\ge n/s=9C$, so \cref{thm:c_expander} shows that $H$ is $\floor{C/8}=\floor{n/(72s)}$-rigid.
\end{proof}

\begin{proof}[Proof of \cref{cor:f_connected}]
  Let $G=(V,E)$ be an $n$-vertex $f$-connected graph.
  We first show that $G$ is a $C/2$-expander.
  Let $\es\ne X\subseteq V$ satisfy $|X|<n/C$.
  If $N(X)=V\sm X$, then $|N(X)|>n/2>(C/2)|X|$.
  Otherwise, applying $f$-connectedness to $A=X\cup N(X)$ and $B=V\sm X$ gives
  \[
    |N(X)|\ge C\min\{|X|,n-|X|-|N(X)|\}.
  \]
  If the minimum is $|X|$, then $|N(X)|\ge C|X|$.
  Otherwise, $|N(X)|\ge n-2|X|>n/2>(C/2)|X|$.
  Thus the first condition in \cref{def:c_expander} holds with expansion parameter $C/2$.

  Now let $X,Y\subseteq V$ be disjoint sets with $|X|,|Y|\ge n/C$, and suppose that there is no edge between them.
  Set $Z=V\sm(X\cup Y)$, $A=X\cup Z$, and $B=Y\cup Z$.
  Then $f$-connectedness gives
  \[
    |Z|\ge C\min\{|X|,|Y|\}\ge n,
  \]
  a contradiction.
  Hence $G$ is a $C/2$-expander.
  Since $C/2>8$ and $n\ge9C\ge9(C/2)$, \cref{thm:c_expander} shows that $G$ is $\floor{C/16}$-rigid.
\end{proof}

\begin{acknowledgements}
  We thank Sahar Diskin for helpful discussions about this work.
\end{acknowledgements}

\bibliography{library}

\end{document}